\documentclass[11pt]{amsart}

\usepackage{amsfonts,amssymb,amsthm,eucal,color,bbm,verbatim,enumitem,graphicx,mleftright}
\usepackage[french,english]{babel}

\usepackage[margin=1.25in]{geometry}

\newtheorem{thm}{Theorem}

\newtheorem{lemma}[thm]{Lemma}
\newtheorem{prop}[thm]{Proposition}

\newcommand{\R}{\mathbb{R}}
\newcommand{\E}{\mathbb{E}}

\newcommand{\Prob}{\mathbb{P}}
\newcommand{\N}{\mathbb{N}}
\newcommand{\n}{\mathcal{N}}
\newcommand{\Z}{\mathbb{Z}}
\newcommand{\C}{\mathbb{C}}

\DeclareMathOperator{\supp}{supp}

\renewcommand{\Re}{\operatorname{Re}}
\renewcommand{\Im}{\operatorname{Im}}
\newcommand{\abs}[1]{\left\vert #1 \right\vert}
\newcommand{\norm}[1]{\left\Vert #1 \right\Vert}

\newcommand{\eps}{\varepsilon}

\DeclareMathOperator{\var}{Var}

\DeclareMathOperator{\Lip}{Lip}

\DeclareMathOperator{\sgn}{sgn}
\newcommand{\Set}[2]{\left\{#1 \mathrel{} \middle| \mathrel{} #2
  \right\}}
\newcommand{\ball}{\mathcal{B}}

\newcommand{\Unitary}[1]{\mathbb{U}\left(#1\right)}
\newcommand{\SUnitary}[1]{\mathbb{SU}\left(#1\right)}
\newcommand{\Orthogonal}[1]{\mathbb{O}\left(#1\right)}
\newcommand{\SOrthogonal}[1]{\mathbb{SO}\left(#1\right)}
\newcommand{\Symplectic}[1]{\mathbb{S}\mathbbm{p}\left(2 #1\right)}
\newcommand{\Circle}{\mathbb{S}^1}

\newcommand{\Mat}[2]{\operatorname{M}_{ #1} \mleft( #2
  \mright)} 
\newcommand{\ind}[1]{\mathbbm{1}_{#1}}

\allowdisplaybreaks[3]

\author{Elizabeth S.\ Meckes and Mark W.\ Meckes}

\address{Department of Mathematics, Applied Mathematics, and
  Statistics, Case Western Reserve University, 10900 Euclid Ave.,
  Cleveland, Ohio 44106, U.S.A.}

\email{elizabeth.meckes@case.edu}

\address{Department of Mathematics, Applied Mathematics, and
  Statistics, Case Western Reserve University, 10900 Euclid Ave.,
  Cleveland, Ohio 44106, U.S.A.}

\email{mark.meckes@case.edu}

\title{Rates of convergence for empirical spectral measures: a soft
  approach}

\begin{document}

\maketitle

\begin{abstract}
 Understanding the limiting behavior of eigenvalues of random matrices
 is the central problem of random matrix theory.  Classical limit
 results are known for many models, and there has been significant
 recent progress in obtaining more quantitative, non-asymptotic
 results.  In this paper, we describe a systematic approach to bounding rates of
convergence and proving tail inequalities for the empirical spectral
measures of a wide variety of random matrix ensembles.  We illustrate
the approach by proving asymptotically almost sure rates of
convergence of the empirical spectral measure in the following
ensembles:  Wigner matrices,
Wishart matrices, Haar-distributed matrices from the compact classical
groups, powers of Haar matrices, randomized sums and random
compressions of Hermitian matrices, a random matrix model for the
Hamiltonians of quantum spin glasses, and finally the complex Ginibre
ensemble.  Many of the results appeared previously and are being
collected and described here as illustrations of the general method; however, some
details (particularly in the Wigner and Wishart cases) are new.

Our approach
makes use of techniques from probability in Banach spaces, in
particular concentration of measure and bounds for suprema of
stochastic processes, in combination with more classical tools from
matrix analysis, approximation theory, and Fourier analysis.  It is
highly flexible, as evidenced by the broad list of examples.   It is
moreover based largely on ``soft'' methods, and
involves little hard analysis. 
\end{abstract}

The most fundamental problem in random matrix theory is to understand
the limiting behavior of the empirical spectral distribution of large
random matrices, as the size tends to infinity.  The first result on
this topic is the famous Wigner semi-circle law, the first version of
which was proved by Wigner in 1955 \cite{Wigner1,Wigner2}.  A random
matrix is called a \emph{Wigner matrix} if it is Hermitian, with
independent entries on and above the diagonal.  Wigner showed that,
under some conditions on the distributions of the entries, the
limiting empirical spectral measure of a (normalized) Wigner matrix is
the semi-circular law $\rho_{sc}$.

Wigner's first version of the semi-circle law gave convergence in
expectation only; i.e., he showed that the expected number of
eigenvalues of a Wigner  matrix in an interval
converged to the value predicted by the semi-circle law, as the size
of the matrix tended to infinity.  His second paper improved this to
convergence ``weakly in probability''.  The analog for random unitary
matrices, namely that their spectral measures converge to the uniform
measure on the circle, seems intuitively obvious;
surprisingly, convergence in mean and weak convergence in probability
were not proved until nearly 40 years after Wigner's original work
\cite{DiSh}.

While these results are fundamental, the limitations of limit theorems
such as these are well known.  Just as the Berry--Esseen theorem and
Hoeffding-type inequalities provide real tools for applications where
the classical central limit theorem only justifies heuristics, it is
essential to improve the classical limit results of random matrix
theory to quantitative approximation results which have content for
large but finite random matrices.  See \cite{DaSz,Vershynin} for
extended discussions of this so-called ``non-asymptotic'' random
matrix theory and its applications.

In this paper, we describe a systematic approach to bounding rates of
convergence and proving tail inequalities for the empirical spectral
measures of a wide variety of random matrix ensembles.  This approach
makes use of techniques from probability in Banach spaces, in
particular concentration of measure and bounds for suprema of
stochastic processes, in combination with more classical tools from
matrix analysis, approximation theory, and Fourier analysis.  Our
approach is highly flexible, and can be used for a wide variety of
types of matrix ensembles, as we will demonstrate in the following
sections.  Moreover, it is based largely on ``soft'' methods, and
involves little hard analysis.  Our approach is
restricted to settings in which there is a concentration of measure
phenomenon; in this sense, it has rather different strengths than the methods used in, for example,
\cite{ErYa,GoTi,TaVu-concentration} and many other works referred to
in those papers.  Those approaches achieve sharper results without requiring a measure
concentration hypothesis, but they require many delicate estimates and
are mainly restricted to random matrices constructed from independent
random variables, whereas our methods have no independence requirements.

The following key observation, a consequence of the classical
Hoffman--Wielandt inequality (see \cite[Theorem VI.4.1]{Bhatia}),
underlies the approach.
\begin{lemma}[see {\cite[Lemma 2.3]{MM-concentration}}] 
  \label{T:Lipschitz} 
  For an $n\times n$ normal matrix $M$ over $\C$, let
  $\lambda_1,\ldots,\lambda_n$ denote the eigenvalues,
  and let $\mu_M$ denote the spectral measure of $M$; i.e.,
  \[
  \mu_M:=\frac{1}{n}\sum_{j=1}^n\delta_{\lambda_j}.
  \]  
  Then
  \begin{enumerate}[label=(\alph*)]
  \item \label{P:integral-is-Lipschitz} if $f:\C\to\R$ is 1-Lipschitz,
    then the map
    \[
    M\longmapsto\int f\ d\mu_M
    \]
    is $\frac{1}{\sqrt{n}}$-Lipschitz, with respect to the
    Hilbert--Schmidt distance on the set of normal matrices; and
  \item \label{P:distance-is-Lipschitz}if $\nu$ is any probability
    measure on $\C$ and $p\in[1,2]$, the map
    \[M\longmapsto W_p(\mu_M,\nu)\]
    is $\frac{1}{\sqrt{n}}$-Lipschitz.
  \end{enumerate}
\end{lemma}

Here $W_p$ denotes the $L_p$-Kantorovich (or Wasserstein) distance on
probability measures on $\C$, defined by
\[
W_p(\mu, \nu) = \left(\inf_\pi \int \abs{x-y}^p \ d\pi(x,y)\right)^{1/p},
\]
where the infimum ranges over probability measures $\pi$ on $\C \times
\C$ with marginals $\mu$ and $\nu$.  The Kantorovich--Rubinstein
theorem (see \cite[Theorem 1.14]{Villani1}) gives that
\[
W_1(\mu, \nu) = \sup_{\abs{f}_L\le 1} \left(\int f\
  d\mu - \int f\ d\nu\right),
\]
where $\abs{f}_L$ denotes the Lipschitz constant of $f$; this
connects part \ref{P:integral-is-Lipschitz} of Lemma \ref{T:Lipschitz}
with estimates on $W_1$.

In many random matrix ensembles of interest there is a concentration
of measure phenomenon, meaning that well-behaved functions are
``essentially constant'', in the sense that they are close to their
means with high probability. A prototype is the following Gaussian
concentration phenomenon (see \cite{Ledoux-book}).

\begin{prop} \label{T:Gaussian-concentration}
  If $F: \R^n \to \R$ is a $1$-Lipschitz function and $Z$ is a
  standard Gaussian random vector in $\R^n$, then
  \[
  \Prob \left[ F(Z) - \E F(Z) \ge t \right] \le e^{-t^2/2}
  \]
  for all $t > 0$.
\end{prop}

Suppose now that $M$ is a random matrix satisfying such a
concentration property.  
Lemma \ref{T:Lipschitz} means that one can obtain a bound on
$W_p(\mu_M, \nu)$ which holds with high probability if one can bound
$\E W_p(\mu_M, \nu)$.  That is, a bound on the \emph{expected}
distance to the limiting measure immediately implies an asymptotically
almost sure bound. The tail estimates coming from measure
concentration are typically exponential or better, and therefore
imply almost sure convergence rates via the Borel--Cantelli
lemma.  

We are thus left with the problem of bounding the expected distance from
the empirical spectral measure $\mu_M$ to some
deterministic reference measure $\nu$. There are two different
methods used for this step, depending on the properties of the
ensemble:
\begin{enumerate}
\item \textbf{Eigenvalue rigidity.} In some ensembles, each of the
  (ordered) individual eigenvalues can be assigned a predicted
  location based on the limiting spectral measure for the ensemble,
  such that all (or at least many) eigenvalues concentrate strongly
  near these predicted locations. In this case $\nu$ is taken to be a
  discrete measure supported on those predicted locations, and the
  concentration allows one to easily estimate $\E W_p(\mu_M, \nu)$.

\smallskip

\item \textbf{Entropy methods.} If instead we set $\nu = \E \mu_M$,
  then the Kantorovich--Rubinstein theorem implies that
  \[
  W_1(\mu_M, \nu) = \sup_{\abs{f}_L\le 1} \left(\int f\
    d\mu_M-\E\int f\ d\mu_M\right),
  \]
  so that $W_1(\mu_M,\nu)$ is the supremum of a centered stochastic
  process indexed by the unit ball of the space of Lipschitz functions
  on $\C$.  In ensembles with a concentration phenomenon for Lipschitz
  functions, part \ref{P:integral-is-Lipschitz} of Lemma
  \ref{T:Lipschitz} translates to an increment condition on this
  stochastic process, which gives a route to bounding its expected
  supremum via classical entropy methods. 
\end{enumerate}

Finally, it may still be necessary to estimate the distance from the
measure $\nu$ to the limiting spectral measure for the random matrix
ensemble. The techniques used to do this vary by the ensemble, but
this is a more classical problem of convergence of a sequence of
deterministic measures to a limit, and any of the many techniques for
obtaining rates of convergence may be useful.

\medskip

Applications of concentration of measure to random matrices date from
at least as long ago as the 1970s; a version of the argument for the
concentration of $W_1(\mu_M,\nu)$ essentially appears in the 2000
paper \cite{GuZe} of Guionnet and Zeitouni.  See
\cite{DaSz,Ledoux-GAFA,Tropp} for surveys of concentration methods in
random matrix theory.

The method of eigenvalue rigidity to bound Kantorovich distances is
particularly suited to situations in which the empirical spectrum is a
determinantal point process; this was first observed in the work of
Dallaporta \cite{Dallaporta1,Dallaporta2}.  The entropy approach to
random Kantorovich distances was introduced in the context of random
projections in \cite{EM-JOTP,EM-GAFA}; it was first applied for
empirical spectral measures in
\cite{MM-compressions,MM-concentration}.  A further abstraction was
given by Ledoux \cite{Ledoux-g2}.

\medskip

\subsection*{Organization} The rest of this paper is a series of
sections sketching some version of the program described above for a
number of random matrix ensembles. Sections \ref{S:Wigner} and section
\ref{S:Wishart} discusses Wigner and Wishart matrices, combining
eigenvalue rigidity arguments of Dallaporta
\cite{Dallaporta1,Dallaporta2} with measure concentration. Section
\ref{S:groups} discusses random matrices drawn uniformly from
classical compact matrix groups, and Section \ref{S:powers} discusses
powers of such matrices; both those sections follow \cite{MM-powers}
and also use the eigenvalue rigidity approach.  The next three
sections use the entropy method: Sections \ref{S:sums} and
\ref{S:compressions} discusses randomized sums and random compressions
of Hermitian matrices, following \cite{MM-concentration}, and Section
\ref{S:qsc} discusses Hamiltonians of quantum spin glasses, following
\cite{BuMe}.  Finally, Section \ref{S:Ginibre}, following
\cite{MM-Ginibre}, demonstrates in case of the complex Ginibre
ensemble, how eigenvalue rigidity alone allows one to carry our much of our
program even without the use of a general concentration phenomenon
together with Lemma \ref{T:Lipschitz}.

\section{Wigner matrices}\label{S:Wigner}

In this section we outline how our approach can be applied to the most
central model of random matrix theory, that of Wigner matrices.  We
begin with the most classical case: the Gaussian Unitary Ensemble
(GUE).  Let $M_n$ be a random $n\times n$ Hermitian matrix, whose
entries $\Set{[M_n]_{jk}}{1 \le j \le k \le n}$ are independent random
variables, such that each $[M_n]_{jj}$ has a $N(0,n^{-1})$
distribution, and each $[M_n]_{jk}$ for $j<k$ has independent real and
imaginary parts, each with a $N(0,(2n)^{-1})$ distribution.  Since
$M_n$ is Hermitian, it has real eigenvalues
$\lambda_1 \le \dots \le\lambda_n$.  Wigner's theorem implies that the
empirical spectral measure
\[
\mu_n = \frac{1}{n} \sum_{j=1}^n \delta_{\lambda_j}
\]
converges to the semicircle law $\rho_{sc}$. The following result
quantifies this convergence.

\begin{thm}\label{T:Wigner}
  Let $M_n$ be as above, and let $\mu_n$ denote its spectral
  measure. Then
  \begin{enumerate}[label=(\alph*)]
  \item
    \label{P:Wigner-expected-distance}
    $\displaystyle \E W_2(\mu_n,\rho_{sc})\le C \frac{\sqrt{\log(n)}}{n},$
  \item \label{P:Wigner-distance-tails}
    $\displaystyle \Prob\left[W_2(\mu_n,\rho_{sc})\ge C
      \frac{\sqrt{\log(n)}}{n}+t\right]\le e^{-n^2 t^2 / 2}$ for all
    $t \ge 0$, and
  \item \label{P:Wigner-as-convergence} with probability 1, for
    sufficiently large $n$, $\displaystyle W_2(\mu_n,\rho_{sc}) \le C'
    \frac{\sqrt{\log(n)}}{n}.$
  \end{enumerate}
\end{thm}
Here and in what follows, symbols such as $c,C,C'$ denote constants
which are independent of dimension.

Part \ref{P:Wigner-expected-distance} of Theorem \ref{T:Wigner} was
proved by Dallaporta in \cite{Dallaporta1} using the eigenvalue
rigidity approach; the proof is 
outlined below.

Lemma \ref{T:Lipschitz} and the Gaussian concentration of measure
property (Proposition \ref{T:Gaussian-concentration}), imply that if
$F$ is a 1-Lipschitz function (with respect to the Hilbert--Schmidt
distance) on the space of Hermitian matrices, then
\begin{equation}
  \label{E:Wigner-concentration}
  \Prob\left[F(M_n)\ge \E F(M_n)+t\right] \le e^{-nt^2/2}
\end{equation}
for all $t \ge 0$. This fact, together with part
\ref{P:distance-is-Lipschitz} of Lemma \ref{T:Lipschitz} and part
\ref{P:Wigner-expected-distance} of Theorem \ref{T:Wigner} now imply
part \ref{P:Wigner-distance-tails}. Finally, part
\ref{P:Wigner-as-convergence} follows from part
\ref{P:Wigner-distance-tails} by the Borel--Cantelli lemma.  So it
remains only to prove part \ref{P:Wigner-expected-distance}.

\medskip

Define $\gamma_j \in \R$ such that
$\rho_{sc}((-\infty,\gamma_j])=\frac{j}{n}$; this is the predicted
location of the $j^{th}$ eigenvalue $\lambda_j$ of $M_n$. The
discretization $\nu_n$ of the semi-circle law $\rho_{sc}$ is given by
\[
\nu_n := \frac{1}{n}\sum_{j=1}^n\delta_{\gamma_j}.
\]
It can be shown that that $W_2(\rho_{sc},\nu_n) \le
\frac{C}{n}$. Furthermore, by the definition of $W_2$,
\[
\E W_2^2(\mu_n,\nu_n) 
\le \frac{1}{n} \sum_{j=1}^n \E\abs{\lambda_j-\gamma_j}^2.
\]
This reduces the proof of part \ref{P:Wigner-expected-distance} to
estimating the latter expectations.

It is a classical fact that the eigenvalues of the GUE form a
determinantal point process with kernel
\[
K_n(x,y) = \sum_{j=0}^n h_j(x) h_j(y) e^{-(x^2 + y^2)/2},
\]
where the $h_j$ are the orthonormalized Hermite polynomials
\cite[Section 6.2]{Mehta}. (The reader is referred to \cite{HKPV06}
for the definition of a determinantal point process.)  The following
is a then a special case of some important general properties of
determinantal point processes \cite[Theorem 7]{HKPV06},
\cite{Gustavsson}.

\begin{prop}
  \label{T:GUE-DPP}
  For each $x \in \R$, let $\mathcal{N}_x$ denote the number of
  eigenvalues of $M_n$ which are less than or equal to $x$. Then
  \[
  \mathcal{N}_x\overset{d}{=}\sum_{i=1}^n\xi_i,
  \]
  where the $\xi_i$ are independent $\{0,1\}$-valued Bernoulli random
  variables.

  Moreover,
  \[
  \E \mathcal{N}_x = \int_{-\infty}^x K_n(u,u) \ du
  \qquad \text{and} \qquad
  \var \mathcal{N}_x = \int_{-\infty}^x \int_x^\infty K_n(u, v)^2 \
  du \ dv.
  \]
\end{prop}

The first part of this result can be combined
with the classical Bernstein inequality to deduce that for each $t >
0$,
\begin{equation*}
  \Prob\left[\abs{\mathcal{N}_x - \E \mathcal{N}_x} > t \right]
  \le 2 \exp \left(-\frac{t^2}{2\sigma_x^2+t}\right),
\end{equation*}
where $\sigma_x^2 = \var \mathcal{N}_x$.  Using estimates on $\E
\mathcal{N}_x$ due to G\"otze and Tikhomirov \cite{GoTi05} and on
$\sigma_x^2$ due to Gustavsson \cite{Gustavsson} (both of which can be
deduced from the second part of Proposition \ref{T:GUE-DPP}), this
implies that for $x\in(-2+\delta,2-\delta)$,
\begin{equation*}
  \Prob\left[\abs{\mathcal{N}_x - n\rho_{sc}((-\infty,x])} > t + C\right]
  \le 2 \exp \left(-\frac{t^2}{2c_\delta\log(n)+t}\right)
\end{equation*}
for each $t \ge 0$. Combining this with the observation
that
\begin{equation*}
\Prob\left[\lambda_j>\gamma_j + t \right]
=\Prob\left[\mathcal{N}_{\gamma_j + t}<j\right],
\end{equation*}
one can deduce, upon integrating by parts, that
\[
\E \abs{\lambda_j - \gamma_j}^2 \le C_\eps \frac{\log(n)}{n^2}
\]
for $j \in [\eps n, (1-\eps) n]$.  This provides the necessary
estimates in the
bulk of the spectrum.  Dallaporta established
similar but weaker bounds for the soft edge of the spectrum using
essentially the last part of Proposition \ref{T:GUE-DPP}, and for the hard
edge using tail estimates due to Ledoux and Rider \cite{LeRi}.
This completes the proof of Theorem \ref{T:Wigner}.

\medskip

The real symmetric counterpart of the GUE is the Gaussian Orthogonal
Ensemble (GOE), whose entries $\Set{[M_n]_{jk}}{1 \le j \le k \le n}$
are independent real random variables, such that each $[M_n]_{jj}$ has
a $N(0,n^{-1})$ distribution, and each $[M_n]_{jk}$ for $j<k$ has a
$N(0,(\sqrt{2}n)^{-1})$ distribution. The spectrum of the GOE does not
form a determinantal point process, but a close distributional
relationship between the eigenvalue counting functions
of the GOE and GUE was found in \cite{FoRa,ORourke}. Using this,
Dallaporta showed that part \ref{P:Wigner-expected-distance} of
Theorem \ref{T:Wigner} also applies to the GOE. Part
\ref{P:Wigner-distance-tails} then follows from the Gaussian
concentration of measure property as before, and part
\ref{P:Wigner-as-convergence} from the Borel--Cantelli lemma.

To move beyond the Gaussian setting, Dallaporta invokes the Tao--Vu
four moment theorem \cite{TaVu1,TaVu2} and a localization theorem due
to Erd\H{o}s, Yau, and Yin \cite{ErYaYi} to extend Theorem
\ref{T:Wigner}\ref{P:Wigner-expected-distance} to random matrices with
somewhat more general entries.  The proofs of these results involve
the kind of hard analysis which it is our purpose to avoid in this
paper.  However, it is straightforward, under appropriate hypotheses,
to extend the measure concentration argument for part
\ref{P:Wigner-distance-tails} of Theorem \ref{T:Wigner}, and we
indicate briefly how this is done.

A probability measure $\mu$ on $\R$ is said to satisfy a quadratic
transportation cost inequality (QTCI) with constant $C > 0$ if
\[
W_2(\mu, \nu) \le \sqrt{C H(\nu \vert \mu)}
\]
for any probability measure $\nu$ which is absolutely continuous with
respect to $\mu$, where $H(\nu \vert \mu)$ denotes relative entropy.

\begin{prop}[{see \cite[Chapter 6]{Ledoux-book}}] 
  \label{T:tci} 
  Suppose that $X_1, \dots, X_n$ are independent random variables
  whose distributions each satisfy a QTCI with constant $C$.  If
  $F: \R^n \to \R$ is a $1$-Lipschitz function, then
  \[ 
  \Prob \left[ F(X) - \E F(X) \ge t \right] \le e^{-t^2/C} 
  \] 
  for all $t > 0$.
\end{prop}

A QTCI is the most general possible hypothesis which implies
subgaussian tail decay, independent of $n$, for Lipschitz functions of
independent random variables; see \cite{Gozlan}.  It holds in
particular for any distribution satisfying a logarithmic Sobolev
inequality, including Gaussian distributions, or a distribution with a
density on a finite interval bounded above and below by positive
constants.  Using Dallaporta's arguments for part
\ref{P:Wigner-expected-distance} and substituting Proposition
\ref{T:tci} in place of the Gaussian concentration phenomenon, we
arrive at the following generalization of Theorem \ref{T:Wigner}.

\begin{thm} 
  \label{T:Wigner-tci}
  Let $M_n$ be a random Hermitian matrix whose entries satisfy each of
  the following:
  \begin{itemize}
  \item The random variables $\left\{\Re M_{jk}\right\}_{1 \le j \le k
      \le n}$ and $\left\{\Im M_{jk}\right\}_{1 \le j < k \le n}$ are
    all independent.
  \item The first four moments of each of these random variables is
    the same as for the GUE (respectively, GOE).
  \item Each of these random variables satisfies a QTCI with constant
    $c n^{-1/2}$.
  \end{itemize}
  Let $\mu_n$ denote the spectral measure of $M_n$. Then
  \begin{enumerate}[label=(\alph*)]
  \item $\displaystyle \E W_2(\mu_n,\rho_{sc})\le C \frac{\sqrt{\log(n)}}{n},$
  \item $\displaystyle \Prob\left[W_2(\mu_n,\rho_{sc})\ge C
      \frac{\sqrt{\log(n)}}{n}+t\right]\le e^{- c n^2 t^2}$ for all
    $t \ge 0$, and
  \item with probability 1, for sufficiently large $n$,
    $\displaystyle  W_2(\mu_n,\rho_{sc}) \le C' \frac{\sqrt{\log(n)}}{n}.$
  \end{enumerate}
\end{thm}

As mentioned above, a QTCI is a minimal assumption to reach exactly
this result by these methods.  A weaker and more classical assumption
would be a Poincar\'e inequality, which implies subexponential decay
for Lipschitz functions, and is the most general hypothesis implying
any decay independent of $n$; see \cite{GoRoSa} and the references
therein.  If the third condition in Theorem \ref{T:Wigner-tci} is
replaced by the assumption of a Poincar\'e inequality with constant
$cn^{-1/2}$, then the same kind of argument leads to an almost sure
convergence rate of order $\frac{\log(n)}{n}$; we omit the details.

\section{Wishart matrices}\label{S:Wishart}

In this section we apply the strategy described in the introduction to
Wishart matrices (i.e., random sample covariance matrices).  Let
$m \ge n$, and let $X$ be an $m\times n$ random matrix with i.i.d.\
entries, and define the Hermitian positive-semidefinite random matrix
\[
S_{m,n}:=\frac{1}{m}X^*X.
\]
We denote the eigenvalues of $S_{m,n}$ by
$0 \le \lambda_1 \le \dots \le \lambda_n$ and the empirical spectral
measure by
\[
\mu_{m,n} = \frac{1}{n} \sum_{j=1}^n \delta_{\lambda_j}.
\]

It was first proved in \cite{MaPa} that, under some moment conditions,
if $\frac{n}{m}\to\rho>0$ as $n,m\to\infty$, then $\mu_{m,n}$
converges to the Marchenko--Pastur law $\mu_\rho$ with parameter
$\rho$, with compactly supported density given by
\[
f_\rho(x)=\frac{1}{2\pi x}\sqrt{(b_\rho-x)(x-a_\rho)},
\]
on $\left(a_\rho,b_\rho\right)$, with $a_\rho=(1-\sqrt{\rho})^2$ and $b_\rho=(1+\sqrt{\rho})^2$.  The
following result quantifies this convergence for many distributions.

\begin{thm}
  \label{T:Wishart}
  Suppose that for each $n$, $0 < c \le \frac{n}{m} \le 1$, and that
  $X$ is an $m \times n$ random matrix whose entries satisfy each of
  the following:
  \begin{itemize}
  \item The random variables
    $\left\{\Re X_{jk}\right\}_{\substack{1 \le j \le m \\ 1 \le k \le
        n}}$
    and
    $\left\{\Im X_{jk}\right\}_{\substack{1 \le j \le m \\ 1 \le k \le
        n}}$ are all independent.
  \item The first four moments of each of these random variables are
    the same as for a standard complex (respectively, real) normal
    random variable.
  \item Each of these random variables satisfies a QTCI with constant
    $C$.
  \end{itemize}
  Let $\rho = \frac{n}{m}$ and let $\mu_{m,n}$ denote the spectral
  measure of $S_{m,n} = \frac{1}{m} X^* X$. Then
  \begin{enumerate}[label=(\alph*)]
  \item \label{P:Wishart-expected-distance}
    $\displaystyle \E W_2(\mu_{m,n},\mu_\rho)\le C \frac{\sqrt{\log(n)}}{n},$
  \item \label{P:Wishart-distance-tails}
    $\displaystyle \Prob\left[W_2(\mu_{m,n},\mu_\rho)\ge C
      \frac{\sqrt{\log(n)}}{n}+t\right]\le e^{- c m \min\{n t^2,
      \sqrt{n} t\}}$ for all $t \ge c \frac{\sqrt{\log(n)}}{n}$, and
  \item \label{P:Wishart-as-convergence} with probability 1, for sufficiently large $n$,
    $\displaystyle  W_2(\mu_{m,n},\mu_\rho) \le C' \frac{\sqrt{\log(n)}}{n}.$
  \end{enumerate}
\end{thm}

Strictly speaking, part \ref{P:Wishart-as-convergence} does not, as
stated, imply almost sure convergence of $\mu_{m,n}$, since $\rho$ and
hence $\mu_{\rho}$ itself depends on $n$. However, if $\rho = \rho(n)$ has a
limiting value $\rho^*$ as $n \to \infty$ (as in the original Marchenko--Pastur
result), then the measures $\mu_{\rho}$ converge to $\mu_{\rho^*}$.
This convergence can easily be quantified, but we will not pursue the
details here.

\begin{proof}
  Part \ref{P:Wishart-expected-distance} was proved by Dallaporta in
  \cite{Dallaporta2}, by the same methods as in Theorem
  \ref{T:Wigner-tci}\ref{P:Wigner-expected-distance} discussed in the
  last section.  First, when the entries of $X$ are complex normal
  random variables (in which $S_{m,n}$ is the unitary Laguerre
  ensemble), the eigenvalues of $S_{m,n}$ form a determinantal point
  process.  This implies an analogue of Proposition \ref{T:GUE-DPP},
  from which eigenvalue rigidity results can be deduced, leading to
  the estimate in part \ref{P:Wishart-expected-distance} in this
  case.  The result is extended to real Gaussian random matrices using
  interlacing results, and to more general distributions using
  versions of the four moment theorem for Wishart random matrices.
  The reader is referred to \cite{Dallaporta2} for the details.

  \medskip

  The proof of part \ref{P:Wishart-distance-tails} is more complicated
  than in the previous section, because the random matrix $S_{m,n}$
  depends quadratically on the independent entries of $X$.  However,
  we can still apply the machinery of measure concentration by using
  the fact that $S_{m,n}$ possesses local Lipschitz behavior, combined
  with a truncation argument. Indeed, if $X,Y$ are $m\times n$
  matrices over $\C$,
  \begin{equation}\begin{split}
      \label{E:not-quite-lipschitz}
      \norm{\frac{1}{m}X^*X-\frac{1}{m}Y^*Y}_{HS}
      & \le \frac{1}{m}\norm{X^*(X-Y)}_{HS} +
      \frac{1}{m}\norm{(X^*-Y^*)Y)}_{HS}\\
      & \le\frac{1}{m}\left(\norm{X}_{op} + \norm{Y}_{op}\right)\norm{X-Y}_{HS},
    \end{split}\end{equation}
  where we have used the facts that both the Hilbert--Schmidt norm
  $\norm{\cdot}_{HS}$ and the operator norm $\norm{\cdot}_{op}$ are
  invariant under conjugation and transposition, and that
  $\norm{AB}_{HS}\le\norm{A}_{op}\norm{B}_{HS}$.

  Thus, for a given $K > 0$, the function
  \[
  X \mapsto \frac{1}{m}X^*X
  \]
  is $\frac{2K}{\sqrt{m}}$-Lipschitz on
  $\Set{X\in\Mat{m,n}{\C}}{\norm{X}_{op}\le K\sqrt{m}},$ and so by
  Lemma \ref{T:Lipschitz}\ref{P:distance-is-Lipschitz}, the function
  \[
F:  X\mapsto W_2(\mu_{m,n},\mu_\rho)
  \]
  is $\frac{2K}{\sqrt{mn}}$-Lipschitz on this set.  We can therefore
  extend $F$ to a $\frac{2K}{\sqrt{mn}}$-Lipschitz function
  $\widetilde{F}:\Mat{m,n}{\C}\to\R$ (cf.\ \cite[Theorem
  3.1.2]{EvGa}); we may moreover assume that $\widetilde{F}(X) \ge 0$
  and
  \begin{equation}
    \label{E:bounded-extension}
    \sup_{X\in\Mat{m,n}{\C}} \widetilde{F}(X)
    =\sup_{\norm{X}_{op}\le K\sqrt{m}} W_2(\mu_{m,n}, \mu_\rho).
  \end{equation}
  Proposition \ref{T:tci} now allows us to control $\widetilde{F}(X)$
  and $\norm{X}_{op}$, which are both Lipschitz functions of $X$.

  First, an elementary discretization argument using Proposition
  \ref{T:tci} (cf.\ \cite[Theorem 5.39]{Vershynin}, or alternatively
  Lemma \ref{T:entropy} below) shows that
  \begin{equation}
    \label{E:op-norm-bound}
    \Prob\left[\norm{X}_{op} > K \sqrt{m} \right] \le 2 e^{-c m}
  \end{equation}
  for some $K, c > 0$. We will use this $K$ in the following.
  
  Next, Proposition \ref{T:tci} implies that
  \begin{equation}
    \label{E:F-tail}
    \Prob\left[\widetilde{F}(X)>t\right] \le C e^{-cmnt^2}
  \end{equation}
  as long as $t \ge 2 \E \widetilde{F}(X)$. Now
  \begin{equation}
    \label{E:E-F-tilde}
    \begin{split}
      \E\widetilde{F}(X) & = \E W_2(\mu_{m,n},\mu_\rho) +
      \E\left[\left(\widetilde{F}(X) -
          W_2(\mu_{m,n},\mu_\rho)\right) \ind{\norm{X}_{op}>K\sqrt{m}} \right] \\
      &\le C\frac{\sqrt{\log(n)}}{n}+\left(\sup_{\norm{X}_{op} \le
          K\sqrt{m}}W_2(\mu_{m,n}, \mu_\rho)\right) \Prob[\norm{X}_{op}>K\sqrt{m} ]
    \end{split}\end{equation}
  by part \ref{P:Wishart-expected-distance} and \eqref{E:bounded-extension}.
  Since $\mu_\rho$ is supported on $[a_\rho,b_\rho]$, and $\mu_{m,n}$
  is supported on
  $\left[0,\norm{\frac{1}{m}XX^*}_{op}\right] =
  \left[0,\frac{1}{m}\norm{X}_{op}^2\right]$,
  \[
  \sup_{\norm{X}_{op}\le K\sqrt{m}} W_2(\mu_{m,n},\mu_\rho)
  \le \max \{b_\rho, K^2\} \le C,
  \]
  and so by \eqref{E:op-norm-bound} and \eqref{E:E-F-tilde},
  \[
  \E\widetilde{F}(X)\le C\frac{\sqrt{\log(n)}}{n}+ C e^{-c m}
  \le C'\frac{\sqrt{\log(n)}}{n}.
  \]
  
  Finally, we have 
  \begin{equation}
    \label{E:W2-Wishart-bound-K}
    \begin{split}
      \Prob\left[W_2(\mu_{m,n},\mu_\rho)>t\right]&\le
      \Prob\left[W_2(\mu_{m,n},\mu_\rho)>t, \norm{X}_{op}\le
        K \sqrt{m}\right]+\Prob\left[\norm{X}_{op}> K\sqrt{m}\right]\\
      &\le\Prob\left[\widetilde{F}(X)>t\right]+\Prob\left[\norm{X}_{op}>
        K\sqrt{m}\right] \\
      & \le C' e^{-cmn t^2}
    \end{split}\end{equation}
  for $c_1 \frac{\sqrt{\log(n)}}{n} \le t \le \frac{c_2}{\sqrt{n}}$ by
  \eqref{E:op-norm-bound} and \eqref{E:F-tail}. We omit the details of
  the similar argument to obtain a subexponential bound for
  $t > \frac{c_2}{\sqrt{n}}$. This concludes the proof of part
  \ref{P:Wishart-distance-tails}.
  
  Part \ref{P:Wishart-as-convergence} follows as before using the
  Borel--Cantelli lemma.
\end{proof}

An alternative approach to quantifying the limiting behavior of the
spectrum of Wishart matrices is to consider the singular values
$0 \le \sigma_1 \le \dots \le \sigma_n$ of $\frac{1}{\sqrt{m}}X$; that
is, $\sigma_j = \sqrt{\lambda_j}$. Lemma \ref{T:Lipschitz} can 
be applied directly in that context, by using the fact that the eigenvalues of
the Hermitian matrix $\begin{bmatrix} 0 & X \\ X^* & 0 \end{bmatrix}$
are $\{\pm \sigma_j\}$. However, if one is ultimately interested in
the eigenvalues $\{ \lambda_j\}$, then translating the resulting
concentration estimates to eigenvalues ends up requiring the same kind
of analysis carried out above.

\section{Uniform random matrices from the compact classical groups}
\label{S:groups}

Each of the compact classical matrix groups $\Orthogonal{n}$,
$\SOrthogonal{n}$, $\Unitary{n}$, $\SUnitary{n}$, $\Symplectic{n}$
possesses a uniform (Haar) probability measure which is invariant under
translation by a fixed group element.  Each of these uniform measures
possesses a concentration of measure property making it amenable to the
program laid out in the introduction; moreover, the eigenvalues of a
random matrix from any of these groups is a determinantal point
process, meaning that the eigenvalue rigidity approach used in Section
\ref{S:Wigner} applies here as well.  The limiting empirical spectral
measure for all of these groups is the uniform probability measure on
the circle, as first shown in \cite{DiSh}.  This convergence is
quantified in the following result, proved in \cite{MM-powers}.

\begin{thm}\label{T:groups}
  Let $M_n$ be uniformly distributed in any of
  $\Orthogonal{n}$, $\SOrthogonal{n}$, $\Unitary{n}$, $\SUnitary{n}$,
  $\Symplectic{n}$, and let $\mu_n$ denote its spectral measure.  Let
  $\mu$ denote the uniform probability measure on the unit circle
  $\Circle\subseteq\C$.  Then
  \begin{enumerate}[label=(\alph*)]
  \item \label{P:groups-expected-distance}
    $\displaystyle \E W_2(\mu_n,\mu)\le C\frac{\sqrt{\log(n)}}{n},$
  \item \label{P:groups-distance-tails}
    $\displaystyle \Prob\left[W_2(\mu_n,\mu) \ge
      C\frac{\sqrt{\log(n)}}{n}+t\right]\le e^{-cn^2t^2},$ and
  \item \label{P:groups-as-convergence} with probability 1, for
    sufficiently large $n$,
    $\displaystyle  W_2(\mu_n,\mu) \le C\frac{\sqrt{\log(n)}}{n}.$
  \end{enumerate}
\end{thm}

We briefly sketch the proof below; for full details, see \cite{MM-powers}.

\medskip

Part \ref{P:groups-expected-distance} is proved using the eigenvalue
rigidity approach described in Section \ref{S:Wigner} for the GUE.  We
first order the eigenvalues of $M_n$ as
$\{e^{i \theta_j}\}_{1 \le j \le n}$ with
$0 \le \theta_1 \le \dots \le \theta_n < 2\pi$, and define the
discretization $\nu_n$ of $\mu$ by
\[
\nu_n := \frac{1}{n} \sum_{j=1}^n \delta_{e^{2\pi i j /n}}.
\]
It is easy to show that $W_2(\mu,\nu_n) \le \frac{C}{n}$, and by the
definition of $W_2$,
\[
\E W_2^2(\mu_n, \nu_n) \le \frac{1}{n} \sum_{j=1}^n \E \abs{e^{i
    \theta_j} - e^{2\pi i j/ n}}^2
\le \frac{1}{n} \sum_{j=1}^n \E \abs{\theta_j - \frac{2\pi j}{n}}^2,
\]
so that part \ref{P:groups-expected-distance} can be proved by
estimating the latter expectations.

For these estimates, as for the GUE, one can make use of the determinantal
structure of the eigenvalue processes of uniformly distributed random
matrices. For the case of the unitary group $\Unitary{n}$, the
eigenvalue angles $\{\theta_j\}$ form a determinantal point process on
$[0,2\pi)$ with kernel
\[
K_n := \frac{\sin \left(\frac{n(x-y)}{2}\right)}{\sin
  \left(\frac{(x-y)}{2}\right)};
\]
this was first proved by Dyson \cite{Dyson}.  The determinantal
structure provides  an analogue
of Proposition \ref{T:GUE-DPP}:

\begin{prop}
  \label{T:groups-DPP}
  For each $0 \le x < 2\pi$, let $\mathcal{N}_x$ denote the number of
  eigenvalues $e^{i \theta_j}$ of $M_n \in \Unitary{n}$ such that
  $\theta_j \le x$. Then
  \begin{equation}\label{E:groups-counting-bernoullis}
  \mathcal{N}_x\overset{d}{=}\sum_{i=1}^n\xi_i,
  \end{equation}
  where the $\xi_i$ are independent $\{0,1\}$-valued Bernoulli random
  variables.

  Moreover,
  \begin{equation}\label{E:groups-means-variances}
  \E \mathcal{N}_x = \int_{0}^x K_n(u,u) \ du
  \qquad \text{and} \qquad
  \var \mathcal{N}_x = \int_{0}^x \int_x^{2\pi} \infty K_n(u, v)^2 \
  du \ dv.
  \end{equation}
\end{prop}

Appropriately modified versions of Proposition \ref{T:groups-DPP} hold for the other
groups as well, due to determinantal structures in those contexts identified by  Katz and Sarnak \cite{KaSa}.

Using \eqref{E:groups-means-variances}, one can estimate $\E \mathcal{N}_x$ and
$\var \mathcal{N}_x$, and then use \eqref{E:groups-counting-bernoullis} and
Bernstein's inequality to deduce
that
\begin{equation}\label{E:groups-Bernstein-with-mean-variance}
  \Prob\left[\abs{\mathcal{N}_x - \frac{nx}{2\pi}} > t + C\right]
  \le 2 \exp \left(-\frac{t^2}{2c\log(n)+t}\right)
\end{equation}
for all $t > 0$. Combining this with the observation
that
\begin{equation*}
\Prob\left[\theta_j>\frac{2\pi j}{n} + t \right]
=\Prob\left[\mathcal{N}_{\frac{2\pi j}{n} + t}<j\right],
\end{equation*}
one can deduce, upon integrating by parts, that
\[
\E \abs{\theta_j - \frac{2\pi j}{n}}^2 \le C \frac{\log(n)}{n^2}
\]
for each $j$, which completes the proof of part
\ref{P:groups-expected-distance}.  Observe that this is made slightly
simpler than the proof of Theorem
\ref{T:Wigner}\ref{P:Wigner-expected-distance} for the GUE by the fact
that all of the eigenvalues of a unitary matrix behave like ``bulk''
eigenvalues.

\medskip

Part \ref{P:groups-distance-tails} of Theorem \ref{T:groups} follows
from part \ref{P:groups-expected-distance} and the following
concentration of measure property of the uniform measure on the
compact classical groups. (There is an additional subtlety in dealing
with the two components of $\Orthogonal{n}$, which can be handled by
conditioning on $\det M_n$.)

\begin{prop}\label{T:groups-concentration}
  Let $G_n$ be one of $\SOrthogonal{n}$, $\Unitary{n}$,
  $\SUnitary{n}$, or $\Symplectic{n}$, and let $F:G_n\to\R$ be
  1-Lipschitz, with respect to either the Hilbert--Schmidt distance or
  the geodesic distance on $G_n$.  Let $M_n$ be a uniformly
  distributed random matrix in $G_n$.  Then
  \[
  \Prob\left[\abs{F(M_n)-\E F(M_n)}>t\right]\le e^{-cnt^2}
  \]
  for every $t > 0$.
\end{prop}

For $\SOrthogonal{n}$, $\SUnitary{n}$, and $\Symplectic{n}$, this 
property goes back to the work of Gromov and Milman \cite{GrMi}; for
the precise version stated here see \cite[Section 4.4]{AnGuZe}.  For
$\Unitary{n}$ (which was not covered by the results of \cite{GrMi}
because its Ricci tensor is degenerate), the concentration in
Proposition \ref{T:groups-concentration} was proved in
\cite{MM-powers}.

Finally, part \ref{P:groups-as-convergence} follows from part
\ref{P:groups-distance-tails} via the Borel-Cantelli lemma, thus
completing the proof of Theorem \ref{T:groups}.

\section{Powers of uniform random matrices}
\label{S:powers}

The approach used with random matrices from the compact classical
groups in the previous section can be readily generalized to powers of
such matrices, as follows.

\begin{thm}
  \label{T:groups-powers}
  Let $M_n$ be uniformly distributed in any of
  $\Orthogonal{n}$, $\SOrthogonal{n}$, $\Unitary{n}$, $\SUnitary{n}$,
  $\Symplectic{n}$.  Let $m\ge 1$,  and let $\mu_{m,n}$ denote the
  spectral measure of $M_n^m$.  Let
  $\mu$ denote the uniform probability measure on the unit circle
  $\Circle\subseteq\C$.  There are universal constants $C,c$ such that
  \begin{enumerate}[label=(\alph*)]
  \item
    \label{P:powers-expected-distance}$\displaystyle \E W_2(\mu_{m,n},\mu)\le
    C\frac{\sqrt{m\left(\log\left(\frac{n}{m}\right)+1\right)}}{n},$
  \item
    \label{P:powers-distance-tails}$\displaystyle \Prob\left[W_2(\mu_{m,n},\mu)\ge
      C\frac{\sqrt{m\left(\log\left(\frac{n}{m}\right)+1\right)}}{n}+t\right]\le
    e^{-cn^2t^2},$ and
  \item \label{P:powers-as-convergence}with probability 1, for
    sufficiently large $n$,
    $\displaystyle  W_2(\mu_{m,n},\mu)\le
    C\frac{\sqrt{m\left(\log\left(\frac{n}{m}\right)+1\right)}}{n}.$
  \end{enumerate}
\end{thm}

In fact, the same proof works for $m>1$ as in the previous section,
because of the following result of Rains \cite{Ra03}.  The result is
stated in the unitary case for simplicity, but analogous results hold
in the other compact classical matrix groups.

\begin{prop}
  \label{T:Rains}
  Let $m\le n$ be fixed. 
  If $M_n$ is uniformly distributed in $\Unitary{n}$, the eigenvalues
  of $M_n^m$ are distributed as those of $m$ independent uniform
  unitary matrices of sizes
  $\left\lfloor \frac{n}{m} \right\rfloor:=\max\left\{k\in\n\mid
    k\le\frac{n}{m}\right\}$
  and
  $\left\lceil \frac{n}{m} \right\rceil:=\min\left\{k\in\n\mid
    k\ge\frac{n}{m}\right\}$,
  such that the sum of the sizes of the matrices is $n$.
\end{prop}

As a consequence, if $\mathcal{N}_x$ is the number of eigenvalues
of $M_n^m$ lying in the arc from 1 to $e^{ix}$, then 
\[
\mathcal{N}_x\overset{d}{=}\sum_{j=0}^{m-1}\mathcal{N}_x^j,
\]
where the $\mathcal{N}_\theta^j$ are the counting functions of $m$
independent random matrices, each uniformly distributed in
$\Unitary{\left\lfloor\frac{n}{m}\right\rfloor}$ or
$\Unitary{\left\lceil\frac{n}{m}\right\rceil}$.  In particular, by
Proposition \ref{T:groups-DPP} $\mathcal{N}_x$ is equal in
distribution to a sum of independent Bernoulli random variables, and
its mean and variance can be estimated using the available estimates
for the individual summands established in the previous section.  One
can thus again apply Bernstein's inequality to obtain eigenvalue
rigidity, leading to a bound on $\E W_2(\mu_{m,n},\mu)$.

Crucially, the concentration phenomenon on the compact classical groups
tensorizes in a dimension-free way: the product of uniform measure on
the $m$ smaller unitary groups above has the same concentration
property as any one of those groups. This is a consequence of the fact
that the uniform measures on the compact classical groups satisfy
logarithmic Sobolev inequalities; see \cite[Section 4.4]{AnGuZe} and
the Appendix of \cite{MM-powers}.  This allows for the full program
laid out in the introduction to be carried out in this case, yielding
Theorem \ref{T:groups-powers} above.

\section{Randomized sums}\label{S:sums}

In this section we show how our approach can be applied to randomized
sums of Hermitian matrices.  In this and the following two sections,
we no longer have a determinantal structure allowing us to use
eigenvalue rigidity.  Instead we will use entropy methods to bound the
expected distance between the empirical spectral measure and its mean.

Let $A_n$ and $B_n$ be fixed $n\times n$ Hermitian matrices, and let
$U_n \in \Unitary{n}$ be uniformly distributed.  Define
\[
M_n := U_n A_n U_n^* + B_n;
\]
the random matrix $M_n$ is the so-called randomized sum of $A_n$ and
$B_n$. This random matrix model has been studied at some length in
free probability theory; the limiting spectral measure was studied
first by Voiculescu \cite{Vo} and Speicher \cite{Sp}, who showed that
if $\{A_n\}$ and $\{B_n\}$ have limiting spectral distributions
$\mu_A$ and $\mu_B$ respectively, then the limiting spectral
distribution of $M_n$ is given by the free convolution
$\mu_A\boxplus\mu_B$. 

The following sharpening of this convergence is a special case of
Theorem 3.8 and Corollary 3.9 of \cite{MM-concentration}; we present
below a slightly simplified version of the argument from that paper.

\begin{thm}
  \label{T:sums} 
  In the setting above, let $\mu_n$ denote the empirical spectral
  measure of $M_n$, and let $\nu_n = \E \mu_n$. Then
  \begin{enumerate}[label=(\alph*)]
  \item \label{P:sums-expected-distance}
    \(\displaystyle  \E W_1(\mu_n, \nu_n) \le
    \frac{C\norm{A_n}_{op}^{2/3}(\norm{A_n}_{op}+\norm{B_n}_{op})^{1/3}}{n^{2/3}}, \)
  \item \label{P:sums-distance-concentration}
    \(\displaystyle  \Prob \left[ W_1(\mu_n, \nu_n) \ge \frac{C\norm{A_n}_{op}^{2/3}(\norm{A_n}_{op}+\norm{B_n}_{op})^{1/3}}{n^{2/3}} + t
    \right] \le e^{ - c n^2 t^2 / \norm{A_n}_{op}^2} \), and
  \item \label{P:sums-as-distance}with probability $1$, for sufficiently
    large $n$, 
    \[ 
    W_1(\mu_{n}, \nu_n) \le
    C'\norm{A_n}_{op}^{2/3}(\norm{A_n}_{op}+\norm{B_n}_{op})^{1/3}
    n^{-2/3}.  
    \]
\end{enumerate}
\end{thm}

In the most typical situations of interest, $\norm{A_n}_{op}$ and
$\norm{B_n}_{op}$ are bounded independently of $n$.  If $\{A_n\}$ and
$\{B_n\}$ have limiting spectral distributions $\mu_A$ and $\mu_B$
respectively, then the rate of convergence of the (deterministic)
measures $\nu_n$ to $\mu_A \boxplus \mu_B$ will depend strongly on the
sequences $\{A_n\}$ and $\{B_n\}$; we will not address that question
here.

The Lipschitz property which is a crucial ingredient of our approach
to prove Theorem \ref{T:sums} is provided by the following lemma.

\begin{lemma}
  \label{T:sums-Lipschitz} 
  For each
  $1$-Lipschitz function $f:\R \to \R$, the maps
  \[
  U_n \mapsto \int f \ d\mu_n
  \qquad \text{and} \qquad
  U_n \mapsto W_1(\mu_n,\nu_n)
  \]
  are $\frac{2\norm{A_n}_{op}}{\sqrt{n}}$-Lipschitz on $\Unitary{n}$.
\end{lemma}

\begin{proof}
  Let $A$ and $B$ be $n \times n$ Hermitian matrices, and let $U, V
  \in \Unitary{n}$. Then it is straightforward to show that
  \[
  \norm{\bigl(UAU^* + B\bigr) - \bigl(VAV^* + B\bigr)}_{HS} 
  \le 2 \norm{A}_{op} \norm{U-V}_{HS}
  \]
  (see \cite[Lemma 3.2]{MM-concentration}). The lemma now follows by
  Lemma \ref{T:Lipschitz}.
\end{proof}

Part \ref{P:sums-distance-concentration} of Theorem \ref{T:sums} now
follows from part \ref{P:sums-expected-distance} using Lemma
\ref{T:sums-Lipschitz} and the concentration of measure phenomenon for
$\Unitary{n}$ (Proposition \ref{T:groups-concentration}), and part
\ref{P:sums-as-distance} follows as usual by the Borel--Cantelli
lemma.  It remains to prove part \ref{P:sums-expected-distance}; as
mentioned above, this is done
using entropy techniques for bounding the supremum of a stochastic
process. 

The following lemma summarizes what is needed here. This fact
is well-known to experts, but we were not able to find an explicit
statement in the literature.  

\begin{lemma}
  \label{T:entropy}
  Suppose that $(V,\norm{\cdot})$ be a finite-dimensional normed space
  with unit ball $\ball(V)$, and that $\Set{X_v}{v\in V}$ is a family
  of centered random variables such that
  \[
  \Prob[\abs{X_u - X_v} \ge t] \le 2 e^{-t^2/K^2 \norm{u-v}^2}
  \]
  for every $t \ge 0$. Then
  \[
  \E \sup_{v \in \ball(V)} X_v \le C K \sqrt{\dim V}.
  \]
\end{lemma}

\begin{proof}
  This can be proved via an elementary $\eps$-net argument, but a
  quicker proof can be given using Dudley's entropy bound (see
  \cite[p.\ 22]{Talagrand} for a statement, and \cite[p.\ 70]{Talagrand}
  and \cite{Dudley} for
  discussions of the history of this result and its name).

  By rescaling it suffices to assume that $K=1$. Let $N(\eps)$ denote
  the number of $\eps$-balls in $V$ needed to cover the unit ball
  $\ball(V)$. A standard volumetric argument (see e.g.\
  \cite[Lemma 5.2]{Vershynin}) shows that
  \( N(\eps) \le (3/\eps)^{\dim V} \)
  for each $0 < \eps < 1$; of course $N(\eps) = 1$ for $\eps \ge
  1$. Then Dudley's bound yields
  \[
  \E \sup_{v \in \ball(V)} X_v \le C \int_0^\infty \sqrt{\log
    (N(\eps))} \ d\eps \le C \sqrt{\dim V} \int_0^1
  \sqrt{\log(3/\eps)} \ d\eps \le C' \sqrt{\dim V}.
  \qedhere
  \]
\end{proof}

To apply this lemma in our setting, denote by
\[
\Lip_0 := \Set{f:\R \to \R}{\abs{f}_L < \infty \text{ and } f(0) =
  0},
\]
so that $\Lip_0$ is a Banach space with norm $\abs{\cdot}_L$.  For
each $f \in \Lip_0$, define the random variable
\begin{equation}
  \label{E:X_f}
X_f := \int f \ d\mu_n - \E \int f \ d\mu_n.
\end{equation}
By the Kantorovich--Rubinstein theorem,
\begin{equation}
  \label{E:W1-Xf}
W_1(\mu_n,\nu_n) = \sup\left\{ X_f : f \in \ball(\Lip_0) \right\}.
\end{equation}
Lemma \ref{T:sums-Lipschitz} and Proposition
\ref{T:groups-concentration} imply that
\begin{equation}
  \label{E:sums-increments}
  \Prob\left[\abs{X_f-X_g} \ge t\right] = \Prob\left[\abs{X_{f-g}} \ge
    t\right]
  \le 2 \exp \left[-\frac{cn^2 t^2}{\norm{A_n}_{op}^2 \abs{f-g}_L^2}\right].
\end{equation}
We would like to appeal to Lemma \ref{T:entropy}, but unfortunately,
$\Lip_0$ is infinite-dimensional.  We can get around this problem with
an additional approximation argument.

Observing that $\mu_n$ is supported on
$[-\norm{M_n}_{op},\norm{M_n}_{op}]$ and
$\norm{M_n}_{op} \le \norm{A_n}_{op} + \norm{B_n}_{op}$, we begin by
replacing $\Lip_0$ with
\[
\Lip_0([-R,R]) := \Set{f:[-R,R] \to \R}{\abs{f}_L < \infty \text{ and } f(0) =
  0},
\]
with $R = \norm{A_n}_{op} + \norm{B_n}_{op}$, for \eqref{E:X_f},
\eqref{E:W1-Xf}, and \eqref{E:sums-increments} above. Now for an
integer $m \ge 1$, let $\Lip_0^m([-R,R])$ be the $2m$-dimensional
space of piecewise affine functions $f \in \Lip_0([-R,R])$ such that
$f$ is affine on each interval
$\left[-R + \frac{(k-1)R}{m},-R+\frac{kR}{m}\right]$ for
$k = 1, \dots, 2m$. Given $f \in \Lip_0([-R,R])$, there is a unique
function $g \in \Lip_0^m([-R,R])$ such that
$g(\frac{jR}{m}) = f(\frac{jR}{m})$ for each integer $j \in [-m,m]$;
and this $g$ satisfies
\[
\abs{g}_L \le \abs{f}_L
\qquad \text{and} \qquad
\norm{f-g}_\infty \le \frac{\abs{f}_LR}{2m}.
\]
Thus by \eqref{E:W1-Xf},
\[
W_1(\mu_n,\nu_n) \le \frac{R}{2m} + \sup \Set{X_g}{g \in \ball(\Lip_0^m([-R,R]))}.
\]
Now by \eqref{E:sums-increments} and Lemma \ref{T:entropy},
\[
\E W_1(\mu_n,\nu_n) \le \frac{R}{2m} + C\frac{\norm{A_n}_{op}\sqrt{m}}{n}.
\]
Part \ref{P:sums-expected-distance} now follows by optimizing over
$m$.  This completes the proof of Theorem \ref{T:sums}.

\medskip

An additional conditioning argument allows one to consider the case
that $A_n$ and $B_n$ are themselves random matrices in Theorem
\ref{T:sums}, assuming a concentration of measure property for their
distributions. We refer to \cite{MM-concentration} for details.

It seems that the entropy method does not usually result in sharp
rates; for example, in \cite{MM-concentration}, we used the entropy
approach for Wigner and Haar-distributed matrices, and the results
were not as strong as those in Sections \ref{S:Wigner} and \ref{S:groups}.  On the other hand,
the entropy method is more widely applicable than the determinantal
point process methods which yielded the results of Sections
\ref{S:Wigner} and \ref{S:groups}.  In addition to the randomized sums
treated in this section, we show in Sections \ref{S:compressions} and
\ref{S:qsc} how the entropy method can be used for random compressions
and for the Hamiltonians of quantum spin glasses.  The paper
\cite{MM-concentration} also used the entropy approach to prove
convergence rates for the empirical spectral measures of the circular
orthogonal ensemble and the circular symplectic ensemble, which we
have omitted from this paper.

\section{Random compressions}\label{S:compressions}
Let $A_n$ be a fixed $n\times n$ Hermitian (respectively, real
symmetric) matrix, and let $U_n$ be uniformly distributed in
$\Unitary{n}$ (respectively, $\Orthogonal{n}$).  Let $P_k$ denote the
projection of $\C^n$ (respectively $\R^n$) onto the span of the first
$k$ standard basis vectors.  Finally, define a random matrix $M_n$ by
\begin{equation}\label{E:compression}
M:=P_k U_n A_n U_n^* P_k^*.
\end{equation}
Then $M_n$ is a compression of $A_n$ to a random $k$-dimensional
subspace. In the case that $\{A_n\}_{n\in\N}$ has a limiting spectral
distribution and $\frac{k}{n}\to\alpha$, the limiting spectral
distribution of $M_n$ can be determined using techniques of free
probability (see \cite{Sp}); the limit is given by a free-convolution
power related to the limiting spectral distribution of $A_n$ and the
value $\alpha$.

For this random matrix model, the program laid out in the introduction
produces the following (cf.\ Theorem 3.5 and Corollary 3.6 in
\cite{MM-concentration}).

\begin{thm}\label{T:compressions}
  In the setting above,  let $\mu_n$ denote the empirical spectral
  distribution of $M_n$, and let $\nu_n=\E \mu_n$.  Then 
  \begin{enumerate}[label=(\alph*)]
  \item \label{P:compressions-expected-distance}  \(\displaystyle 
    \E W_1(\mu_n, \nu_n) \le \frac{C\norm{A_n}_{op}}{((kn)^{1/3}},
    \)
  \item \label{P:compressions-distance-concentration}
    \(\displaystyle  \Prob \left[ W_1(\mu_n, \nu_n) \ge
      \frac{C\norm{A_n}_{op}}{(kn)^{1/3}} + t \right] \le e^{ - c kn
      t^2/\norm{A_n}_{op}^2} \), and
  \item \label{P:compressions-as-distance} with probability $1$, for
    sufficiently large $n$,
    \(\displaystyle 
    W_1(\mu_n, \nu_n) \le C' \norm{A_n}_{op} (kn)^{-1/3}.
    \)
\end{enumerate}
\end{thm}

The proof is essentially identical to the one in the previous section;
the $k$-dependence in the bounds is a consequence of the fact that
$k$, not $n$, is the size of the matrix when Lemma \ref{T:Lipschitz}
is applied. As with Theorem \ref{T:sums}, an additional conditioning
argument allows one to consider the case that $A_n$ is random, with
distribution satisfying a concentration of measure property.

\section{Hamiltonians of quantum spin glasses}
\label{S:qsc}

In this section we consider the following random matrix model for the
Hamiltonian of a quantum spin glass: let
$\{Z_{a,b,j}\}_{\substack{1\le a,b\le 3\\1\le j\le n}}$ be independent
standard Gaussian random variables, and define the $2^n\times 2^n$
random Hermitian matrix $H_n$ by
\begin{equation}
  \label{E:H_n}
  H_n := \frac{1}{\sqrt{9n}} \sum_{j=1}^n \sum_{a,b=1}^3 Z_{a,b,j}
  \sigma^{(a)}_j \sigma^{(b)}_{j+1},
\end{equation}
where for $1\le a\le 3$,
\[
\sigma^{(a)}_j := I_n^{\otimes(j-1)}\otimes\sigma^{(a)}\otimes
I_2^{\otimes(n-j)},
\]
with $I_2$ denoting the $2\times 2$ identity matrix, $\sigma^{(a)}$
denoting the $2\times 2$ matrices
\[
\sigma^{(1)}:=\begin{bmatrix}0&1\\1&0\end{bmatrix}\qquad
\sigma^{(2)}:=\begin{bmatrix}0&-i\\i&0\end{bmatrix}\qquad
\sigma^{(3)}:=\begin{bmatrix}1&0\\0&-1\end{bmatrix},
\]
and the labeling cyclic so that
$\sigma^{(b)}_{n+1}:=\sigma^{(b)}_{1}.$ The random matrix $H_n$ acts
on the space $(\C^2)^{\otimes n}$ of $n$ distinguishable qubits; the
specific structure of $H_n$ above corresponds to nearest neighbor
interaction on a circle of qubits.

If $\mu_n$ denotes the empirical spectral measure of $H_n$, then the
ensemble average $\nu_n = \E \mu_n$ is known in this context as the
density of states measure $\mu_n^{DOS}$. Recently, Keating, Linden and
Wells \cite{KLW} showed that $\mu_n^{DOS}$ converges weakly to
Gaussian, as $n\to\infty$; i.e., they showed that the empirical
spectral measure of $H_n$ converges to Gaussian in expectation.  The
paper \cite{KLW} gives a similar treatment for more general
collections of (still independent) coupling coefficients, and more
general coupling geometries than that of nearest-neighbor
interactions.  In more recent work, Erd\H{os} and Schr\"oder \cite{ES}
have considered still more general coupling geometries, and found a
sharp transition in the limiting behavior of the density of states
measure depending on the size of the maximum degree of the underlying
graph, relative to its number of edges.

The following result, essentially proved in \cite{BuMe}, quantifies
this convergence.

\begin{thm}\label{T:quantum}
  Let $\mu_n$ be the spectral measure of $H_n$ and let $\gamma$ denote
  the standard Gaussian distribution on $\R$.  Then
  \begin{enumerate}[label=(\alph*)]
  \item \label{P:quantum-expected-distance}\(\displaystyle \E W_1(\mu_n,\gamma)\le \frac{C}{n^{1/6}},\)
  \item
    \label{P:quantum-distance-concentration}\(\displaystyle
    \Prob\left[W_1(\mu_n,\gamma)\ge\frac{C}{n^{1/6}}+t\right]\le
    e^{-9nt^2/2},\) and
  \item \label{P:quantum-as-distance}with probability 1, for all
    sufficiently large $n$,
    \[
    W_1(\mu_n,\gamma)\le \frac{C'}{n^{1/6}}.
    \]
\end{enumerate}
\end{thm}

Because the coefficients $Z_{a,b,j}$ in \eqref{E:H_n} are taken to be
i.i.d.\ Gaussian random variables, the Gaussian concentration of
measure phenomenon (Proposition \ref{T:Gaussian-concentration}) can be
combined with Lemma \ref{T:Lipschitz} to carry out a version of the
approach used in the cases of random sums and random compressions
(Sections \ref{S:sums} and \ref{S:compressions}). The following lemma
provides the necessary link between Lemma \ref{T:Lipschitz} and
Proposition \ref{T:Gaussian-concentration} for this random matrix
model.

\begin{lemma}\label{T:quantum-lipschitz}
Let $\mathbf{x}=\{x_{a,b,j}\}\in\R^{9n}$ (with, say, lexicographic
ordering), and assume that $n\ge 3$.  Define $H_n(\mathbf{x})$ by 
\[
H_n(\mathbf{x}) := \frac{1}{3\sqrt{n}} \sum_{a,b=1}^3
\sum_{j=1}^nx_{a,b,j} \sigma^{(a)}_j \sigma^{(b)}_{j+1}.
\]
Then the map $\mathbf{x}\mapsto H_n$ is
$\frac{2^{n/2}}{3\sqrt{n}}$-Lipschitz.
\end{lemma}

Lemma \ref{T:quantum-lipschitz} and Lemma
\ref{T:Lipschitz}\ref{P:distance-is-Lipschitz} together show that
\[
\mathbf{x}\mapsto W_1(\mu_n,\gamma)
\]
is a $\frac{1}{3\sqrt{n}}$-Lipschitz function of $\mathbf{x}$.  Part
\ref{P:quantum-distance-concentration} of Theorem \ref{T:quantum} then
follows from part \ref{P:quantum-expected-distance} and Proposition
\ref{T:Gaussian-concentration}, and part \ref{P:quantum-as-distance}
follows by the Borel--Cantelli lemma.

The proof of part \ref{P:quantum-expected-distance} has two main
components.  First, $W_1(\mu_n,\E\mu_n)$ is estimated via the approach
used in Sections \ref{S:sums} and \ref{S:compressions}: Lemma
\ref{T:quantum-lipschitz}, Lemma
\ref{T:Lipschitz}\ref{P:integral-is-Lipschitz}, and Proposition
\ref{T:Gaussian-concentration} show that the stochastic process
\[
X_f:=\int f \ d\mu_n-\E\int f\ d\mu_n
\]
satisfies a subgaussian increment condition as in Lemma
\ref{T:entropy}, which can then be used to show that.
\[
\E W_1(\mu_n, \E \mu_n) \le \frac{C}{n^{1/6}}.
\]

Second, the convergence in expectation proved in \cite{KLW} was done
via a pointwise estimate of the difference between the characteristic
functions of $\E \mu_n $ and $\gamma$; this estimate can be parlayed
into an estimate on $W_1(\E\mu_n,\gamma)$ via Fourier analysis.  This
is carried out in detail in \cite{BuMe} for the bounded-Lipschitz
distance; a similar argument shows that
\[
W_1(\E\mu_n,\gamma)\le\frac{C}{n^{1/6}},
\]
completing the proof of Theorem \ref{T:quantum}.

\section{The complex Ginibre ensemble}
\label{S:Ginibre}

Let $G_n$ be an $n \times n$ random matrix with i.i.d.\ standard
complex Gaussian entries; $G_n$ is said to belong to the \emph{complex
  Ginibre ensemble}.  It was first established by Mehta that if
$\mu_n$ is the empirical spectral measure of $\frac{1}{\sqrt{n}} G_n$,
then as $n \to \infty$, $\mu_n$ converges to the circular law; i.e.,
to the uniform measure $\mu$ on the unit disc
$D:=\Set{z \in \C}{\abs{z} \le 1}$.

This is the one ensemble we treat in which the general concentration
of measure approach does not apply.  The issue is that while there is
a concentration phenomenon for the i.i.d.\ Gaussian entries of $G_n$,
the spectral measure of a nonnormal matrix ($G_n$ is nonnormal with
probability 1) is not a Lipschitz function of the matrix.
Nevertheless, the eigenvalue process of $G_n$ is a determinantal point
process, and so some of the techniques used above are still available.
We sketch the basic idea below; full details can be found in
\cite{MM-Ginibre}

\medskip

The eigenvalues of $G_n$ form a determinantal point process on $\C$
with the kernel
\begin{equation} \label{E:DPP}
  \begin{split}
    K(z,w) & = \frac{1}{\pi} e^{-(\abs{z}^2 + \abs{w}^2) / 2}
    \sum_{k=0}^{n-1} \frac{(z\overline{w})^k}{k!}.
  \end{split}
\end{equation}
This means that in principle, the determinantal approach to eigenvalue
rigidity used in the case of the GUE (Section \ref{S:Wigner}) and of
the compact classical groups (Section \ref{S:groups}) can be used for
this model.  A challenge, however, is the lack of an obvious order
on the eigenvalues of an arbitrary matrix over $\C$; without one,
there is no hope to assign predicted locations around which the
individual eigenvalues concentrate.  We therefore impose an order on
$\C$ which is well-adapted for our purposes; we refer to this as the
\emph{spiral order}.  Specifically, the linear order $\prec$ on $\C$
is defined by making $0$ initial, and for nonzero $w, z \in \C$, we
declare $w \prec z$ if any of the following holds:
\begin{itemize}
\item $\lfloor \sqrt{n} \abs{w} \rfloor < \lfloor \sqrt{n} \abs{z} \rfloor$.
\item $\lfloor \sqrt{n} \abs{w} \rfloor = \lfloor \sqrt{n} \abs{z} \rfloor$ and
  $\arg w < \arg z$.
\item $\lfloor \sqrt{n} \abs{w} \rfloor = \lfloor \sqrt{n} \abs{z} \rfloor$,
  $\arg w = \arg z$, and $\abs{w} \ge \abs{z}$.
\end{itemize}
Here we are using the convention that $\arg z \in (0,2\pi]$.   

We order the eigenvalues according to $\prec$: first the eigenvalues
in the disc of radius $\frac{1}{\sqrt{n}}$ are listed in order of
increasing argument, then the ones in the annulus with inner radius
$\frac{1}{\sqrt{n}}$ and outer radius $\frac{2}{\sqrt{n}}$ in order of
increasing argument, and so on.
We then define predicted locations $\tilde{\lambda}_j$ for (most of)
the eigenvalues based on the spiral order: $\tilde{\lambda}_1 = 0$,
$\{\tilde{\lambda}_2, \tilde{\lambda}_3, \tilde{\lambda}_4\}$ are
$\frac{1}{\sqrt{n}}$ times the $3^{\mathrm{rd}}$ roots of unity (in
increasing order with respect to $\prec$), the next five are
$\frac{2}{\sqrt{n}}$ times the $5^{\mathrm{th}}$ roots of unity, and
so on.  Letting $\nu_n$ denote the normalized counting measure
supported on the $\{\tilde{\lambda}_j\}$, it is easy to show that
\[
W_2(\nu_n, \mu) \le \frac{C}{\sqrt{n}}.
\]
(In fact, there is a slight modification for about $\sqrt{n\log(n)}$
of the largest eigenvalues, the details of which we will not discuss
here.)

The same type of argument as in the earlier determinantal cases gives
a Bernstein-type inequality for the eigenvalue counting function on an
initial segment with respect to the spiral order, which in turn leads
to eigenvalue rigidity for most of the eigenvalues.  The largest
eigenvalues can be treated with a more elementary argument, leading
via the usual coupling argument to the bound
\[
\E W_2(\mu_n,\nu_n)\le C \left(\frac{\log(n)}{n}\right)^{1/4}.
\]
(One can deduce a slightly tighter bound for $\E W_p(\mu_n,\nu_n)$ for
$1 \le p < 2$, and a weaker one for $p > 2$.)

In this setting we cannot argue that the concentration of
$W_1(\mu_n,\mu)$ is immediate from general concentration properties of
the ensemble, but the eigenvalue rigidity itself can be used as a
substitute.  Indeed,
\[
W_2(\mu_n,\nu_n)^2 \le \frac{1}{n} \sum_{j=1}^n \abs{\lambda_j-\tilde{\lambda}_j}^2,
\]
and so 
\[
\Prob\left[W_2(\mu_n,\nu_n)^2 > t\right]
\le \Prob\left[\sum_{j=1}^n\abs{\lambda_j-\tilde{\lambda}_j}^2 > nt\right]
\le \sum_{j=1}^n
\Prob\left[\abs{\lambda_j-\tilde{\lambda}_j}^2 > t\right]. 
\]
For most of the eigenvalues the eigenvalue rigidity about
$\tilde{\lambda}_j$ is strong enough to bound this quite sharply; as
before, for about $\sqrt{n\log(n)}$ of the largest eigenvalues a
more trivial bound is used. Since this approach does
not produce a particularly clean tail inequality for
$W_2(\mu_n,\nu_n)$, we will instead simply state the almost-sure
convergence rate which follows by the Borel--Cantelli lemma.

\begin{thm}\label{T:as-rate}
  Let $\mu_n$ denote the empirical spectral measure of
  $\frac{1}{\sqrt{n}} G_n$, and let $\mu$ denote the uniform measure
  on the unit disc in $\C$. Then with probability $1$, for
  sufficiently large $n$,
  \[
  W_2(\mu_n, \mu) \le C \frac{\sqrt{\log n}}{n^{1/4}}.
  \]
  \end{thm}

\nocite{Chatterjee, Kargin, ChLe, NgWa}
\section*{Acknowledgements}

This research was partially supported by grants from the U.S. National Science Foundation
(DMS-1308725 to E.M.)  and the Simons
Foundation (\#315593 to M.M.).  This paper is an expansion of the first-named
author's talk at the excellent workshop ``Information Theory and
Concentration Phenomena'' at the Institute for Mathematics and its
Applications, as part of the IMA Thematic Year on Discrete Structures:
Analysis and Applications.  The authors thank the IMA for its hospitality.

\bibliographystyle{plain}
\bibliography{ima}

\begin{thebibliography}{10}

\bibitem{AnGuZe}
G.~W. Anderson, A.~Guionnet, and O.~Zeitouni.
\newblock {\em An Introduction to Random Matrices}, volume 118 of {\em
  Cambridge Studies in Advanced Mathematics}.
\newblock Cambridge University Press, Cambridge, 2010.

\bibitem{Bhatia}
R.~Bhatia.
\newblock {\em Matrix Analysis}, volume 169 of {\em Graduate Texts in
  Mathematics}.
\newblock Springer-Verlag, New York, 1997.

\bibitem{BuMe}
D.~Buzinski and E.~S. Meckes.
\newblock Almost sure convergence in quantum spin glasses.
\newblock {\em J. Math. Phys.}, 56(12), 2015.

\bibitem{Chatterjee}
S.~Chatterjee.
\newblock Concentration of {H}aar measures, with an application to random
  matrices.
\newblock {\em J. Funct. Anal.}, 245(2):379--389, 2007.

\bibitem{ChLe}
S.~Chatterjee and M.~Ledoux.
\newblock An observation about submatrices.
\newblock {\em Electron. Commun. Probab.}, 14:495--500, 2009.

\bibitem{Dallaporta1}
S.~Dallaporta.
\newblock Eigenvalue variance bounds for {W}igner and covariance random
  matrices.
\newblock {\em Random Matrices Theory Appl.}, 1(3):1250007, 28, 2012.

\bibitem{Dallaporta2}
S.~Dallaporta.
\newblock Eigenvalue variance bounds for covariance matrices.
\newblock {\em Markov Process. Related Fields}, 21(1):145--175, 2015.

\bibitem{DaSz}
K.~R. Davidson and S.~J. Szarek.
\newblock Local operator theory, random matrices and {B}anach spaces.
\newblock In {\em Handbook of the Geometry of {B}anach Spaces, {V}ol. {I}},
  pages 317--366. North-Holland, Amsterdam, 2001.

\bibitem{DiSh}
P.~Diaconis and M.~Shahshahani.
\newblock On the eigenvalues of random matrices.
\newblock {\em J. Appl. Probab.}, 31A:49--62, 1994.
\newblock Studies in applied probability.

\bibitem{Dudley}
R.~Dudley.
\newblock {V. N. S}udakov's work on expected suprema of {G}aussian processes.
\newblock To appear in the proceedings of High Dimensional Probability VII.

\bibitem{Dyson}
F.~J. Dyson.
\newblock Correlations between eigenvalues of a random matrix.
\newblock {\em Comm. Math. Phys.}, 19:235--250, 1970.

\bibitem{ES}
L.~Erd{\H{o}}s and D.~Schr{\"o}der.
\newblock Phase transition in the density of states of quantum spin glasses.
\newblock {\em Math. Phys. Anal. Geom.}, 17(3-4):441--464, 2014.

\bibitem{ErYa}
L.~Erd{\H{o}}s and H.-T. Yau.
\newblock Universality of local spectral statistics of random matrices.
\newblock {\em Bull. Amer. Math. Soc. (N.S.)}, 49(3):377--414, 2012.

\bibitem{ErYaYi}
L.~Erd{\H{o}}s, H.-T. Yau, and J.~Yin.
\newblock Rigidity of eigenvalues of generalized {W}igner matrices.
\newblock {\em Adv. Math.}, 229(3):1435--1515, 2012.

\bibitem{EvGa}
L.~C. Evans and R.~F. Gariepy.
\newblock {\em Measure theory and fine properties of functions}.
\newblock Studies in Advanced Mathematics. CRC Press, Boca Raton, FL, 1992.

\bibitem{FoRa}
P.~J. Forrester and E.~M. Rains.
\newblock Interrelationships between orthogonal, unitary and symplectic matrix
  ensembles.
\newblock In {\em Random matrix models and their applications}, volume~40 of
  {\em Math. Sci. Res. Inst. Publ.}, pages 171--207. Cambridge Univ. Press,
  Cambridge, 2001.

\bibitem{GoTi}
F.~G{\"o}tze and A.~Tikhomirov.
\newblock Optimal bounds for convergence of expected spectral distributions to
  the semi-circular law.
\newblock To appear in \textit{Probab. Theory Related Fields}.

\bibitem{GoTi05}
F.~G{\"o}tze and A.~Tikhomirov.
\newblock The rate of convergence for spectra of {GUE} and {LUE} matrix
  ensembles.
\newblock {\em Cent. Eur. J. Math.}, 3(4):666--704 (electronic), 2005.

\bibitem{Gozlan}
N.~Gozlan.
\newblock A characterization of dimension free concentration in terms of
  transportation inequalities.
\newblock {\em Ann. Probab.}, 37(6):2480--2498, 2009.

\bibitem{GoRoSa}
N.~Gozlan, C.~Roberto, and P.-M. Samson.
\newblock From dimension free concentration to the {P}oincar\'e inequality.
\newblock {\em Calc. Var. Partial Differential Equations}, 52(3-4):899--925,
  2015.

\bibitem{GrMi}
M.~Gromov and V.~D. Milman.
\newblock A topological application of the isoperimetric inequality.
\newblock {\em Amer. J. Math.}, 105(4):843--854, 1983.

\bibitem{GuZe}
A.~Guionnet and O.~Zeitouni.
\newblock Concentration of the spectral measure for large matrices.
\newblock {\em Electron. Comm. Probab.}, 5:119--136 (electronic), 2000.

\bibitem{Gustavsson}
J.~Gustavsson.
\newblock Gaussian fluctuations of eigenvalues in the {GUE}.
\newblock {\em Ann. Inst. H. Poincar\'e Probab. Statist.}, 41(2):151--178,
  2005.

\bibitem{HKPV06}
J.~B. Hough, M.~Krishnapur, Y.~Peres, and B.~Vir{\'a}g.
\newblock Determinantal processes and independence.
\newblock {\em Probab. Surv.}, 3:206--229, 2006.

\bibitem{Kargin}
V.~Kargin.
\newblock A concentration inequality and a local law for the sum of two random
  matrices.
\newblock {\em Probab. Theory Related Fields}, 154(3-4):677--702, 2012.

\bibitem{KaSa}
N.~M. Katz and P.~Sarnak.
\newblock {\em Random Matrices, {F}robenius Eigenvalues, and Monodromy},
  volume~45 of {\em American Mathematical Society Colloquium Publications}.
\newblock American Mathematical Society, Providence, RI, 1999.

\bibitem{KLW}
J.~P. Keating, N.~Linden, and H.~J. Wells.
\newblock Spectra and eigenstates of spin chain {H}amiltonians.
\newblock {\em Comm. Math. Phys.}, 338(1):81--102, 2015.

\bibitem{Ledoux-book}
M.~Ledoux.
\newblock {\em The Concentration of Measure Phenomenon}, volume~89 of {\em
  Mathematical Surveys and Monographs}.
\newblock American Mathematical Society, Providence, RI, 2001.

\bibitem{Ledoux-GAFA}
M.~Ledoux.
\newblock Deviation inequalities on largest eigenvalues.
\newblock In {\em Geometric Aspects of Functional Analysis}, volume 1910 of
  {\em Lecture Notes in Math.}, pages 167--219. Springer, Berlin, 2007.

\bibitem{Ledoux-g2}
M.~Ledoux.
\newblock $\gamma_2$ and {$\Gamma_2$}.
\newblock Unpublished note, available at\\
  \texttt{http://perso.math.univ-toulouse.fr/ledoux/files/2015/06/gGamma2.pdf},
  2015.

\bibitem{LeRi}
M.~Ledoux and B.~Rider.
\newblock Small deviations for beta ensembles.
\newblock {\em Electron. J. Probab.}, 15:no. 41, 1319--1343, 2010.

\bibitem{MaPa}
V.~A. Mar{\v{c}}enko and L.~A. Pastur.
\newblock Distribution of eigenvalues in certain sets of random matrices.
\newblock {\em Mat. Sb. (N.S.)}, 72 (114):507--536, 1967.

\bibitem{EM-JOTP}
E.~Meckes.
\newblock Approximation of projections of random vectors.
\newblock {\em J. Theoret. Probab.}, 25(2):333--352, 2012.

\bibitem{EM-GAFA}
E.~Meckes.
\newblock Projections of probability distributions: a measure-theoretic
  {D}voretzky theorem.
\newblock In {\em Geometric Aspects of Functional Analysis}, volume 2050 of
  {\em Lecture Notes in Math.}, pages 317--326. Springer, Heidelberg, 2012.

\bibitem{MM-compressions}
E.~S. Meckes and M.~W. Meckes.
\newblock Another observation about operator compressions.
\newblock {\em Proc. Amer. Math. Soc.}, 139(4):1433--1439, 2011.

\bibitem{MM-concentration}
E.~S. Meckes and M.~W. Meckes.
\newblock Concentration and convergence rates for spectral measures of random
  matrices.
\newblock {\em Probab. Theory Related Fields}, 156(1-2):145--164, 2013.

\bibitem{MM-powers}
E.~S. Meckes and M.~W. Meckes.
\newblock Spectral measures of powers of random matrices.
\newblock {\em Electron. Commun. Probab.}, 18:no. 78, 13, 2013.

\bibitem{MM-Ginibre}
E.~S. Meckes and M.~W. Meckes.
\newblock A rate of convergence for the circular law for the complex {G}inibre
  ensemble.
\newblock {\em Ann. Fac. Sci. Toulouse Math. (6)}, 24(1):93--117, 2015.

\bibitem{Mehta}
M.~L. Mehta.
\newblock {\em Random Matrices}, volume 142 of {\em Pure and Applied
  Mathematics (Amsterdam)}.
\newblock Elsevier/Academic Press, Amsterdam, third edition, 2004.

\bibitem{NgWa}
S.~Ng and M.~Walters.
\newblock A method to derive concentration of measure bounds on {M}arkov
  chains.
\newblock {\em Electron. Comm. Probab.}, 20(95), 2015.

\bibitem{ORourke}
S.~O'Rourke.
\newblock Gaussian fluctuations of eigenvalues in {W}igner random matrices.
\newblock {\em J. Stat. Phys.}, 138(6):1045--1066, 2010.

\bibitem{Ra03}
E.~M. Rains.
\newblock Images of eigenvalue distributions under power maps.
\newblock {\em Probab. Theory Related Fields}, 125(4):522--538, 2003.

\bibitem{Sp}
R.~Speicher.
\newblock Free convolution and the random sum of matrices.
\newblock {\em Publ. Res. Inst. Math. Sci.}, 29(5):731--744, 1993.

\bibitem{Talagrand}
M.~Talagrand.
\newblock {\em Upper and Lower Bounds for Stochastic Processes: Modern Methods
  and Classical Problems}, volume~60 of {\em Ergebnisse der Mathematik und
  ihrer Grenzgebiete. 3. Folge.}
\newblock Springer, Heidelberg, 2014.

\bibitem{TaVu2}
T.~Tao and V.~Vu.
\newblock Random matrices: universality of local eigenvalue statistics up to
  the edge.
\newblock {\em Comm. Math. Phys.}, 298(2):549--572, 2010.

\bibitem{TaVu1}
T.~Tao and V.~Vu.
\newblock Random matrices: universality of local eigenvalue statistics.
\newblock {\em Acta Math.}, 206(1):127--204, 2011.

\bibitem{TaVu-concentration}
T.~Tao and V.~Vu.
\newblock Random matrices: sharp concentration of eigenvalues.
\newblock {\em Random Matrices Theory Appl.}, 2(3):1350007, 31, 2013.

\bibitem{Tropp}
J.~A. Tropp.
\newblock An introduction to matrix concentration inequalities.
\newblock {\em Foundations and Trends in Machine Learning}, 8(1-2), 2015.

\bibitem{Vershynin}
R.~Vershynin.
\newblock Introduction to the non-asymptotic analysis of random matrices.
\newblock In {\em Compressed Sensing}, pages 210--268. Cambridge Univ. Press,
  Cambridge, 2012.

\bibitem{Villani1}
C.~Villani.
\newblock {\em Topics in Optimal Transportation}, volume~58 of {\em Graduate
  Studies in Mathematics}.
\newblock American Mathematical Society, Providence, RI, 2003.

\bibitem{Vo}
D.~Voiculescu.
\newblock Limit laws for random matrices and free products.
\newblock {\em Invent. Math.}, 104(1):201--220, 1991.

\bibitem{Wigner1}
E.~Wigner.
\newblock Characteristic vectors of bordered matrices with infinite dimensions.
\newblock {\em Ann. of Math. (2)}, 62:548--564, 1955.

\bibitem{Wigner2}
E.~Wigner.
\newblock On the distribution of the roots of certain symmetric matrices.
\newblock {\em Ann. of Math. (2)}, 67:325--327, 1958.

\end{thebibliography}

\end{document}